 \documentclass[a4paper,12pt]{article}
 \usepackage[T1]{fontenc}
\usepackage[utf8]{inputenc}
\usepackage{authblk}

\usepackage{graphicx,amsmath,amssymb,euscript,amsfonts}
\usepackage{textcomp}
\graphicspath{{images/}}
\usepackage[english]{babel}
\usepackage{amsthm,amssymb,amsmath}
\usepackage{epsfig}
\usepackage{color}
\usepackage[usenames,dvipsnames]{pstricks}
\usepackage{tikz}
\input epsf

\newtheorem{theorem}{Theorem}

\newtheorem{conjecture}[theorem]{Conjecture}

\newtheorem{lemma}[theorem]{Lemma}

\newtheorem{Remark}[theorem]{Remark}

\begin{document}
\title{A partial proof of the Brouwer's conjecture}

\markright{Abbreviated Article Title}

\author [a]{Slobodan Filipovski \thanks{Corresponding author; {slobodan.filipovski@famnit.upr.si}}}

\affil[a] {\small \it FAMNIT, University of Primorska, Koper, Slovenia}

 \date{}
\maketitle

\begin{abstract} Let $G$ be a simple graph with $n$ vertices and $m$ edges and let $k$ be a natural number such that $k\leq n.$ Brouwer conjectured that the sum of the $k$ largest Laplacian eigenvalues of $G$ is at most $m+{k+1 \choose 2}.$ In this paper we prove that this conjecture is true for simple $(m,n)$-graphs where $n\leq m\leq \frac{\sqrt{3}-1}{4}(n-1)n$ and $k\in  \left[ \sqrt[3]{\frac{8m^{2}}{n-1}+4mn+n^{2}}, n\right].$ Moreover, we prove that the conjecture is true for all simple $(m,n)$-graphs where $k (\leq n)$ is a natural number from the interval $\left[\sqrt{2n-2m+2\sqrt{2m^{2}+mn(n-1)}},1+\frac{8m^{2}}{n^{2}(n-1)}+\frac{4m}{n}\right].$

\end{abstract}

\maketitle






\section{Introduction} \quad Let $G$ be a simple graph with vertex set $V(G)$ and edge set $E(G)$ and let $|V(G)|=n$ and $|E(G)|=m.$
The degree of a vertex $v\in V(G)$ denoted by $d(v)$, is the number of neighbours of $v$. 
The Laplacian matrix of $G$ is the $n \times n$ matrix $L(G)=[l_{ij}]$ defined as $L(G)=D(G)-A(G)$, where $D(G)$ is the diagonal matrix of vertex degrees of the graph $G$, and $A(G)$
is the adjacency matrix of $G$.
It is well known that $L(G)$ is a positive semidefinite matrix, consequently its eigenvalues are nonnegative real numbers. The eigenvalues of $L(G)$ are called the Laplacian
eigenvalues of $G$ and are denoted by $\mu_{1}(G)\geq \mu_{2}(G)\geq \ldots \geq \mu_{n}(G)$. Since each row sum of $L(G)$ is $0$, it holds $\mu_{n}(G)=0.$
In this paper, we study the sum $S_{k}(G)=\sum_{i=1}^{k} \mu_{i}(G)$ for $1\leq k \leq n$. Brouwer \cite{bro}  has conjectured the following upper bound for the sum $S_{k}(G).$

\begin{conjecture}\cite{bro} Let $G$ be a graph with vertices $n$ and edges $m$. Then
$$S_{k}(G)\leq m+  { k+1 \choose 2}$$
for $k=1, 2, \ldots , n.$
\end{conjecture}
Based on a computer check, Brouwer \cite{bro} showed that the Conjecture 1 is true for all graphs with at most $10$ vertices. There have been many partial proofs of this conjecture using specific methods from spectral graph theory. For $k=1$, the conjecture follows from the well-known inequality 
$\mu_{1}(G)\leq |V(G)|=n,$ see in \cite{god}. The conjecture has been proved to be true for any graphs with $k=2, n-2, n-3$ , see \cite{chen, hae},  trees \cite{hae}, threshold graphs \cite{hel}, unicyclic graphs \cite{du}, bicyclic graphs \cite{du, wang}, tricyclic graphs having no pendant vertices \cite{kum, wang}, regular graphs \cite{ber, may}, split graphs \cite{may}, cographs \cite{ber, may}, planar graphs when $k\geq 11$ and bipartite graphs when $k\geq \sqrt{32n}$ \cite{coo}. Also the cases $k=n$ and $k=n-1$ are straightforward.
Another related upper bound was given by Zhou in \cite{bo}: 
$$S_{k}(G)\leq \frac{2mk+\sqrt{mk(n-k-1)(n^{2}-n-2m)}}{n-1}$$
where $1\leq k \leq n-2$ and $m=|E(G)|.$

In this paper we prove that the Brouwer's conjecture is true for all simple $(m,n)$-graphs where $n\leq m\leq \frac{\sqrt{3}-1}{4}(n-1)n$ and $k\in  \left[ \sqrt[3]{\frac{8m^{2}}{n-1}+4mn+n^{2}}, n\right].$ Moreover, we prove that this conjecture is true for all simple $(m,n)$-graphs where $k\leq n$ is a natural number from the interval $\left[\sqrt{2n-2m+2\sqrt{2m^{2}+mn(n-1)}},1+\frac{8m^{2}}{n^{2}(n-1)}+\frac{4m}{n}\right].$

\section{A partial proof of the conjecture}

\quad Let $J$ be all-ones matrix of size $n$ and let $L^{'}=J-L.$ Since $(L^{'})^{\texttt{T}}=J^{\texttt{T}}-L^{\texttt{T}}=J-L=L^{'},$ we obtain that $L^{'}$ is a symmetric matrix.
Let $\mu_{1}\geq \mu_{2}\geq \ldots \geq \mu_{n}$ be the eigenvalues of $L$, and let $\mu_{1}^{'}\geq \mu_{2}^{'}\geq \ldots \geq \mu_{n}^{'}$ be the eigenvalues of $L^{'},$ in the decreasing order. Moreover, let $\lambda_{1}\geq \lambda_{2}\geq \ldots \geq \lambda_{n}$ be the eigenvalues of $J$. It is well-known that $\lambda_{1}(J)=n$ and $\lambda_{i}(J)=0$, for $i=2, 3, \ldots , n.$
Applying the Weyl's inequality \cite{weyl} we get the following result. 
\begin{lemma}
\begin{itemize}
\item [a)] $\mu_{i}+\mu^{'}_{n-i+2}\leq0 \;\;\text{for}\;\; i=2, 3, \ldots, n.$
\item [b)] $\mu_{1}+\mu_{1}^{'}\geq n.$
\end{itemize}
\end{lemma}
\begin{proof}
\begin{itemize}
\item[a)] From Weyl's inequality we get:
\begin{equation}\label{prva} \mu_{i}(L)+\mu^{'}_{j}(L^{'})\leq \lambda_{i+j-n}(L+L^{'})=\lambda_{i+j-n}(J).
\end{equation}
Setting $j=n-i+2$ in (\ref{prva}) we get the inequality:
\begin{equation}\label{druga}
\mu_{i}+\mu^{'}_{n-i+2}\leq \lambda_{2}(J)=0, \;\;\text{for}\;\; i=2, 3, \ldots, n.
\end{equation}
\item[b)] Since $\text{trace}(L)=\sum_{i=1}^{n} \mu_{i},$ $\text{trace}(L^{'})=\sum_{i=1}^{n} \mu_{i}^{'}$ and 
$\text{trace}(L)+\text{trace}(L^{'})=\text{trace}(J)=n$ we have 
$$\mu_{1}+\mu_{1}^{'}+\sum_{i=2}^{n}(\mu_{i}+\mu_{n-i+2}^{'})=n.$$
From a) we get $\mu_{1}+\mu_{1}^{'}\geq n.$
\end{itemize}
\end{proof}
The \emph{first Zagreb index} for a graph with a vertex set $\{v_{1}, v_{2}, \ldots, v_{n}\}$ and vertex degrees $d_i=\text{deg}(v_{i})$ for $i=1, \ldots, n$ is defined by
$$M_{1}(G)=\sum_{i=1}^{n} d_{i}^{2}.$$ The following upper bound for $M_{1}$, obtained by de Caen in \cite{caen}, will be used in our proof.
\begin{lemma} \cite{caen} Let $G=(V, E)$ be a graph with $n$ vertices and $m$ edges. Then
$$M_{1}(G)=\sum_{i=1}^{n} d_{i}^{2} \leq \frac{2m^{2}}{n-1}+mn-2m.$$
Moreover, if $G$ is connected, then equality holds if and only if $G$ is either a star $K_{1, n-1}$ or a complete graph $K_{n}.$
\end{lemma}
In the next theorem we give an upper bound for $S_{k}(G)$ in terms of $m, n$ and $k.$

\begin{theorem} Let $G$ be a simple graph on $n$ vertices and $m$ edges. Then
$$S_{k}(G)=\sum_{i=1}^{k} \mu_{k} \leq n +\sqrt{(k-1)\left(\frac{2m^{2}}{n-1}+mn+\frac{n^{2}}{4}\right)}.$$
\end{theorem}

\begin{proof} From $\text{trace}(L\cdot L^{\texttt{T}})=\sum_{i=1}^{n}\sum_{j=1}^{n}l_{ij}^{2}$ and $L=L^{\texttt{T}}$ we obtain
\begin{equation}\label{eigenvalues1}\sum_{i=1}^{n}\mu_{i}^{2}=\text{trace}(L^{2})=\text{trace}(L\cdot L^{\texttt{T}})=\sum_{i=1}^{n}\sum_{j=1}^{n}l_{ij}^{2}.
\end{equation}
Analogously we get
\begin{equation}\label{eigenvalues2}\sum_{i=1}^{n}(\mu_{i}^{'})^{2}=\text{trace}((L^{'})^{2})=\text{trace}(L^{'}\cdot (L^{'})^{\texttt{T}})=\sum_{i=1}^{n}\sum_{j=1}^{n}(l_{ij}^{'})^{2}.
\end{equation}
Combining (\ref{eigenvalues1}) and (\ref{eigenvalues2}) we get:
$$\sum_{i=1}^{n}\mu_{i}^{2}+\sum_{i=1}^{n}(\mu_{i}^{'})^{2}=\sum_{i=1}^{n}\sum_{j=1}^{n}l_{ij}^{2}+\sum_{i=1}^{n}\sum_{j=1}^{n}(l_{ij}^{'})^{2}=\sum_{i=1}^{n}\sum _{j=1}^{n} (l_{ij}^{2}+(1-l_{ij})^{2})$$
$$=\sum_{i=1}^{n}\sum_{j=1}^{n}(2 l_{ij}^{2}-2 l_{ij}+1)=\left(\sum_{i=1}^{n} 2d_{i}^{2}-\sum_{i=1}^{n}2d_{i}+n\right)+\left(\sum_{i=1}^{n}\sum_{j=1, j\neq i}^{n}(2 l_{ij}^{2}-2l_{ij}+1)\right)$$
$$=(2M_{1}-4m+n)+((4+4+\ldots+4)+n^{2}-n)$$
$$=2M_{1}-4m+n+2m\cdot 4 +n^{2}-n=2M_{1}+4m+n^{2}.$$
Let $t=\frac{\mu_{1}}{n}.$ Since $\mu_{1}\leq n$ we get $0<t\leq 1.$ The inequality $\mu_{1}+\mu_{1}^{'}\geq n$ implies
\begin{equation}\label{largest}
\mu_{1}^{2}+(\mu_{1}^{'})^{2}\geq n^{2}t^{2}+n^{2}(1-t)^{2}=n^{2}(2t^{2}-2t+1).
\end{equation}
Moreover, from $\mu_{i}+\mu^{'}_{n-i+2}\leq 0, \;\;\text{for}\;\; i=2, 3, \ldots, n,$ and from $\mu_{i}\geq 0$ we easily conclude that $\mu_{i}^{2} \leq (\mu_{n-i+2}^{'})^{2}.$
Therefore we have
$$2M_{1}+4m+n^{2}=\sum_{i=1}^{n}\mu_{i}^{2}+\sum_{i=1}^{n} (\mu_{i}^{'})^{2}=(\mu_{1}^{2}+(\mu_{1}^{'})^{2})+\sum_{i=2}^{n}(\mu_{i}^{2}+(\mu_{n-i+2}^{'})^{2})$$
$$\geq (1-2t(1-t))n^{2}+2\sum_{i=2}^{n}\mu_{i}^{2}.$$
From the last inequality we get
\begin{equation}\label{glav}
\sum_{i=2}^{k}\mu_{i}^{2}\leq \sum_{i=2}^{n} \mu_{i}^{2} \leq \frac{2M_{1}+4m+n^{2}}{2}-\frac{(1-2t(1-t))n^{2}}{2}.
\end{equation}
Now, based on the inequality (\ref{glav}) and applying the inequality between the arithmetic and quadratic means for the nonnegative eigenvalues $\mu_{2}, \mu_{3}, \ldots, \mu_{n}$ we get
\begin{equation}\label{arit}
\sum_{i=2}^{k} \mu_{i} \leq \sqrt{(k-1)(\mu_{2}^{2}+\ldots+\mu_{k}^{2})}\leq \sqrt{(k-1) \left(\frac{2M_{1}+4m+n^{2}}{2}-\frac{(1-2t(1-t))n^{2}}{2}\right)}.
\end{equation}
Hence we get
\begin{equation} \label{conclusion}
S_{k}(G)=\sum_{i=1}^{k} \mu_{i}=\mu_{1}+\sum_{i=2}^{n} \mu_{i}\leq nt+ \sqrt{(k-1) \left(M_{1}+2m+n^{2}t(1-t)\right)}.
\end{equation}
Setting $M_{1}\leq \frac{2m^{2}}{n-1}-2m+mn$ in (\ref{conclusion}) we obtain the following inequality
\begin{equation}\label{izvod}
\sum_{i=1}^{k}\mu_{i} \leq nt+\sqrt{(k-1)\left(\frac{2m^{2}}{n-1}+mn+n^{2}t(1-t)\right)}.
\end{equation}
In the end, from $t\leq 1$ and $t(1-t)\leq \frac{1}{4}$ we have
$$S_{k}(G)=\sum_{i=1}^{k}\mu_{i} \leq n+\sqrt{(k-1)\left(\frac{2m^{2}}{n-1}+mn+\frac{n^{2}}{4}\right)}.$$

\end{proof}

\begin{theorem} Let $G$ be a simple graph on $n$ vertices and $m$ edges such that\\ $n\leq m\leq \frac{\sqrt{3}-1}{4}(n-1)n.$ 
The Brouwer's conjecture holds for all natural numbers  $k$ from the interval $ \left[\sqrt[3]{\frac{8m^{2}}{n-1}+4mn+n^{2}}, n\right].$ 
\end{theorem}
\begin{proof} The bound in Theorem 4 and the assumption $n\leq m$ yield
\begin{equation}\label{bitna}
S_{k}(G) \leq n+\sqrt{(k-1)\left(\frac{2m^{2}}{n-1}+mn+\frac{n^{2}}{4}\right)}<m+\sqrt{k\left(\frac{2m^{2}}{n-1}+mn+\frac{n^{2}}{4}\right)}.
\end{equation}
Now, we are going to show that for  $k \in \left[\sqrt[3]{\frac{8m^{2}}{n-1}+4mn+n^{2}}, n\right]$, it holds  $$m+\sqrt{k\left(\frac{2m^{2}}{n-1}+mn+\frac{n^{2}}{4}\right)}\leq m+{k+1 \choose 2}.$$
Since $\frac{k^{2}}{2}<{k+1 \choose 2}, $ it suffices to prove that
$$k\left(\frac{2m^{2}}{n-1}+mn+\frac{n^{2}}{4}\right)\leq\frac{k^{4}}{4},$$
which obviously holds for $k\geq \sqrt[3]{\frac{8m^{2}}{n-1}+4mn+n^{2}}.$\\
In the end, we verify that for $m\leq \frac{\sqrt{3}-1}{4}(n-1)n$ it occurs $\sqrt[3]{\frac{8m^{2}}{n-1}+4mn+n^{2}}<n.$
The inequality $\sqrt[3]{\frac{8m^{2}}{n-1}+4mn+n^{2}}<n$ is equivalent to the quadratic inequality (in $m$)
\begin{equation}\label{quadratic}
8m^{2}+4mn(n-1)-n^{2}(n-1)^{2}<0.
\end{equation}
Solving the inequality (\ref{quadratic}) we get $m \in \left(\frac{(-\sqrt{3}-1)}{4}n(n-1), \frac{(\sqrt{3}-1)}{4}n(n-1)\right).$

\end{proof}

\begin{Remark} We illustrate the above theorem for simple graphs on $100$ vertices. From the condition $n\leq m\leq \frac{\sqrt{3}-1}{4}n(n-1)$ we
have $100\leq m\leq 1811.$ Based on the restrictions for $k$, that is, $k\in  \left[\sqrt[3]{\frac{8m^{2}}{n-1}+4mn+n^{2}}, n\right],$ 
 in the table below we consider particular examples and give the corresponding values of $k$ for which the Brouwer's conjecture is true.
\begin{center}
\begin{tabular}{|c|c|}
  \hline
    m & k \\
     \hline
    100 & [38, 100] \\
     \hline
    200 & [46, 100] \\
     \hline
    300 & [52, 100] \\
     \hline
    400 & [57, 100] \\
     \hline
    500 & [62, 100] \\
     \hline
    600 & [66, 100] \\
  \hline
  700 & [70, 100] \\
  \hline
\end{tabular}
\end{center}
\end{Remark}
Note that the Theorem 5 concerns the simple graphs with $n$ vertices and with at most $\frac{(\sqrt{3}-1)}{4}n(n-1)$ edges. This is approximately $\frac{2}{5}$ of all connected graphs on $n$ vertices. The next theorem extends the set of graphs for which the Brouwer's conjecture is true.

\begin{theorem} Let $G$ be a simple graph on $n$ vertices and $m$ edges such that\\ $\sqrt{2n-2m+2\sqrt{2m^{2}+mn(n-1)}}<1+\frac{8m^{2}}{n^{2}(n-1)}+\frac{4m}{n}.$\\
The Brouwer's conjecture is true for any natural number $k\leq n$ from the interval  $$\left[\sqrt{2n-2m+2\sqrt{2m^{2}+mn(n-1)}},1+\frac{8m^{2}}{n^{2}(n-1)}+\frac{4m}{n}\right].$$
\end{theorem}
\begin{proof} Let $g(t)=nt+\sqrt{(k-1)\left(\frac{2m^{2}}{n-1}+mn+n^{2}t(1-t)\right)}$ where $0\leq t\leq 1.$ Then $g^{'}(t)=n+\frac{n^{2}(k-1)(1-2t)}{2\sqrt{(k-1)\left(\frac{2m^{2}}{n-1}+mn+n^{2}t(1-t)\right)}}.$
If $0\leq t \leq \frac{1}{2}$, then it is clear that $g^{'}(t)>0.$ \\ Now, let $t>\frac{1}{2}.$
We obtain
\begin{equation}\label{hm}
g^{'}(t)>0 \Leftrightarrow n>\frac{n^{2}(k-1)(2t-1)}{2\sqrt{(k-1)\left(\frac{2m^{2}}{n-1}+mn+n^{2}t(1-t)\right)}}.
\end{equation}
The inequality (\ref{hm}) is equivalent to
\begin{equation}\label{nerav}
\frac{8m^{2}}{n-1}+4mn+4n^{2}t(1-t)\geq n^{2}(k-1)(2t-1)^{2}.
\end{equation}
We easily observe that the inequality in (\ref{nerav}) holds since $2t-1\leq 1$, $n^{2}t(1-t)\geq 0$ and $k\leq 1+\frac{8m^{2}}{n^{2}(n-1)}+\frac{4m}{n}.$\\
Since $g^{'}(t)>0$, we get that $g(t)$ is an increasing function on the interval $[0, 1]$. Recalling the inequality (\ref{izvod}) we have
$$S_{k}(G)\leq g(t)\leq  g(1)=n+\sqrt{(k-1)\left(\frac{2m^{2}}{n-1}+mn\right)}.$$
Let us suppose that $k\geq \sqrt{2n-2m+2\sqrt{2m^{2}+mn(n-1)}}$.  We get $2m+k^{2}\geq 2n+2\sqrt{2m^{2}+mn(n-1)}.$ Hence
$$m+{k+1 \choose 2} >m+\frac{k^{2}}{2}\geq n+\sqrt{2m^{2}+mn(n-1)}=n+\sqrt{(n-1)\left(\frac{2m^{2}}{n-1}+mn\right)}$$
$$\geq n+\sqrt{(k-1)\left(\frac{2m^{2}}{n-1}+mn\right)}=g(1)\geq S_{k}(G).$$

\end{proof}

\begin{Remark} We illustrate the above theorem for simple graphs on $100$ vertices. From the condition $\sqrt{2n-2m+2\sqrt{2m^{2}+mn(n-1)}}<1+\frac{8m^{2}}{n^{2}(n-1)}+\frac{4m}{n}$ we
get $m\geq 1468.$ Based on the restrictions for $k$, in the table below we consider particular examples and give the corresponding values of $k$ for which the Brouwer's conjecture is true.
\begin{center}
\begin{tabular}{|c|c|}
  \hline
    m & k \\
     \hline
    1500 & [78,79] \\
     \hline
    1600 & [79,85] \\
     \hline
    1700 & [80, 92] \\
     \hline
    1800 & [81, 99] \\
     \hline
    1900 & [82, 100] \\
     \hline
    2000 & [83, 100] \\
  \hline
  2100 & [84, 100] \\
  \hline
\end{tabular}
\end{center}

\end{Remark}
\textbf{Acknowledgement:} 
This work is supported in part by the Slovenian Research Agency (research program P1-0285 and research projects N1-0210, J1-3001, J1-3002, J1-3003, and J1-4414)

\end{document}